\documentclass[12pt]{amsart}

\usepackage{etoolbox}

\makeatletter
\let\old@tocline\@tocline
\let\section@tocline\@tocline
\newcommand{\subsection@dotsep}{4.5}
\newcommand{\subsubsection@dotsep}{4.5}
\patchcmd{\@tocline}
  {\hfil}
  {\nobreak
     \leaders\hbox{$\m@th
        \mkern \subsection@dotsep mu\hbox{.}\mkern \subsection@dotsep mu$}\hfill
     \nobreak}{}{}
\let\subsection@tocline\@tocline
\let\@tocline\old@tocline

\patchcmd{\@tocline}
  {\hfil}
  {\nobreak
     \leaders\hbox{$\m@th
        \mkern \subsubsection@dotsep mu\hbox{.}\mkern \subsubsection@dotsep mu$}\hfill
     \nobreak}{}{}
\let\subsubsection@tocline\@tocline
\let\@tocline\old@tocline

\let\old@l@subsection\l@subsection
\let\old@l@subsubsection\l@subsubsection

\def\@tocwriteb#1#2#3{%
  \begingroup
    \@xp\def\csname #2@tocline\endcsname##1##2##3##4##5##6{%
      \ifnum##1>\c@tocdepth
      \else \sbox\z@{##5\let\indentlabel\@tochangmeasure##6}\fi}%
    \csname l@#2\endcsname{#1{\csname#2name\endcsname}{\@secnumber}{}}%
  \endgroup
  \addcontentsline{toc}{#2}%
    {\protect#1{\csname#2name\endcsname}{\@secnumber}{#3}}}%

\newlength{\@tocsectionindent}
\newlength{\@tocsubsectionindent}
\newlength{\@tocsubsubsectionindent}
\newlength{\@tocsectionnumwidth}
\newlength{\@tocsubsectionnumwidth}
\newlength{\@tocsubsubsectionnumwidth}
\newcommand{\settocsectionnumwidth}[1]{\setlength{\@tocsectionnumwidth}{#1}}
\newcommand{\settocsubsectionnumwidth}[1]{\setlength{\@tocsubsectionnumwidth}{#1}}
\newcommand{\settocsubsubsectionnumwidth}[1]{\setlength{\@tocsubsubsectionnumwidth}{#1}}
\newcommand{\settocsectionindent}[1]{\setlength{\@tocsectionindent}{#1}}
\newcommand{\settocsubsectionindent}[1]{\setlength{\@tocsubsectionindent}{#1}}
\newcommand{\settocsubsubsectionindent}[1]{\setlength{\@tocsubsubsectionindent}{#1}}

\renewcommand{\l@section}{\section@tocline{1}{\@tocsectionvskip}{\@tocsectionindent}{}{\@tocsectionformat}}%
\renewcommand{\l@subsection}{\subsection@tocline{2}{\@tocsubsectionvskip}{\@tocsubsectionindent}{}{\@tocsubsectionformat}}%
\renewcommand{\l@subsubsection}{\subsubsection@tocline{3}{\@tocsubsubsectionvskip}{\@tocsubsubsectionindent}{}{\@tocsubsubsectionformat}}%
\newcommand{\@tocsectionformat}{}
\newcommand{\@tocsubsectionformat}{}
\newcommand{\@tocsubsubsectionformat}{}
\expandafter\def\csname toc@1format\endcsname{\@tocsectionformat}
\expandafter\def\csname toc@2format\endcsname{\@tocsubsectionformat}
\expandafter\def\csname toc@3format\endcsname{\@tocsubsubsectionformat}
\newcommand{\settocsectionformat}[1]{\renewcommand{\@tocsectionformat}{#1}}
\newcommand{\settocsubsectionformat}[1]{\renewcommand{\@tocsubsectionformat}{#1}}
\newcommand{\settocsubsubsectionformat}[1]{\renewcommand{\@tocsubsubsectionformat}{#1}}
\newlength{\@tocsectionvskip}
\newcommand{\settocsectionvskip}[1]{\setlength{\@tocsectionvskip}{#1}}
\newlength{\@tocsubsectionvskip}
\newcommand{\settocsubsectionvskip}[1]{\setlength{\@tocsubsectionvskip}{#1}}
\newlength{\@tocsubsubsectionvskip}
\newcommand{\settocsubsubsectionvskip}[1]{\setlength{\@tocsubsubsectionvskip}{#1}}

\patchcmd{\tocsection}{\indentlabel}{\makebox[\@tocsectionnumwidth][l]}{}{}
\patchcmd{\tocsubsection}{\indentlabel}{\makebox[\@tocsubsectionnumwidth][l]}{}{}
\patchcmd{\tocsubsubsection}{\indentlabel}{\makebox[\@tocsubsubsectionnumwidth][l]}{}{}

\newcommand{\@sectypepnumformat}{}
\renewcommand{\contentsline}[1]{%
  \expandafter\let\expandafter\@sectypepnumformat\csname @toc#1pnumformat\endcsname%
  \csname l@#1\endcsname}
\newcommand{\@tocsectionpnumformat}{}
\newcommand{\@tocsubsectionpnumformat}{}
\newcommand{\@tocsubsubsectionpnumformat}{}
\newcommand{\setsectionpnumformat}[1]{\renewcommand{\@tocsectionpnumformat}{#1}}
\newcommand{\setsubsectionpnumformat}[1]{\renewcommand{\@tocsubsectionpnumformat}{#1}}
\newcommand{\setsubsubsectionpnumformat}[1]{\renewcommand{\@tocsubsubsectionpnumformat}{#1}}
\renewcommand{\@tocpagenum}[1]{%
  \hfill {\mdseries\@sectypepnumformat #1}}

\let\oldappendix\appendix
\renewcommand{\appendix}{%
  \leavevmode\oldappendix%
  \addtocontents{toc}{%
    \protect\settowidth{\protect\@tocsectionnumwidth}{\protect\@tocsectionformat\sectionname\space}%
    \protect\addtolength{\protect\@tocsectionnumwidth}{2em}}%
}
\makeatother

\makeatletter
\settocsectionnumwidth{2em}
\settocsubsectionnumwidth{2.5em}
\settocsubsubsectionnumwidth{3em}
\settocsectionindent{1pc}%
\settocsubsectionindent{\dimexpr\@tocsectionindent+\@tocsectionnumwidth}%
\settocsubsubsectionindent{\dimexpr\@tocsubsectionindent+\@tocsubsectionnumwidth}%
\makeatother

\settocsectionvskip{5pt}
\settocsubsectionvskip{0pt}
\settocsubsubsectionvskip{0pt}
    
\settocsectionformat{\bfseries}
\settocsubsectionformat{\mdseries}
\settocsubsubsectionformat{\mdseries}
\setsectionpnumformat{\bfseries}
\setsubsectionpnumformat{\mdseries}
\setsubsubsectionpnumformat{\mdseries}

\let\oldtableofcontents\tableofcontents
\renewcommand{\tableofcontents}{%
  \vspace*{-\linespacing}
  \oldtableofcontents}

\usepackage{amsmath, amssymb, amsthm, amsfonts, marginnote}
\usepackage{mathtools}
\usepackage{tikz-cd} 

\input xy
\xyoption{all}

\usepackage[bookmarks=true, bookmarksopen=true,%
bookmarksdepth=3,bookmarksopenlevel=2,%
colorlinks=true,%
linkcolor=blue,%
citecolor=blue,%
filecolor=blue,%
menucolor=blue,%
urlcolor=blue]{hyperref}

\theoremstyle{plain}
\newtheorem{theorem}{Theorem}[section]
\newtheorem{lemma}[theorem]{Lemma}

\newtheorem{proposition}[theorem]{Proposition}

\newtheorem{corollary}[theorem]{Corollary}

\theoremstyle{definition}
\newtheorem{defn}[theorem]{Definition}
\newtheorem{definition}[theorem]{Definition}

\newtheorem{remark}[theorem]{Remark}

\newtheorem{question}[theorem]{Question}

\usepackage[margin=1in,marginparwidth=0.8in, marginparsep=0.1in]{geometry}

\def\CC{\mathbb{C}}

\def\RR{\mathbb{R}}
\def\R{\RR}

\def\ZZ{\mathbb{Z}}

\def\R{\mathbb{R}}
\def\Z{\mathbb{Z}}

\newcommand\cC{\mathcal{C}}

\newcommand\cV{\mathcal{V}}

\newcommand\frc{\mathfrak{c}}

\newcommand\frg{\mathfrak{g}}

\newcommand{\isom}{\stackrel{\sim}{\to}}

\newcommand{\wt}[1]{\widetilde{#1}}

\newcommand\quash[1]{}

\newcommand\ol{\overline}

\newcommand\beq{\begin{equation}}
\newcommand\eeq{\end{equation}}

\usepackage{graphicx}

\newcommand{\mush}{\mathit{\mu sh}}

\usepackage{hieroglf}

\title{Invariance of microsheaves on stable Higgs bundles}

\dedicatory{}
\author{David Nadler}
\thanks{}
\address{Department of Mathematics, UC Berkeley, Evans Hall, Berkeley CA 94720, USA}
\email{nadler@math.berkeley.edu}
\author{Vivek Shende}
\thanks{}
\address{Centre for Quantum Mathematics, SDU, Campusvej 55, 5230 Odense M, Denmark $\qquad \qquad$
\& Department of Mathematics, UC Berkeley, Evans Hall, Berkeley CA 94720, USA}
\email{vivek.vijay.shende@gmail.com}
\date{}
\subjclass[2010]{}

\begin{document}


\begin{abstract}
The spectral side of the (conjectural) Betti geometric Langlands correspondence concerns sheaves 
on the character stack of an algebraic curve; in particular, the categories in question are manifestly 
invariant under deformations of the curve.  By contrast the same invariance is certainly not manifest, and is
presently not known, for their automorphic counterparts, in particular because 
the singularities of the global nilpotent cone may vary significantly with the complex structure of the curve.  

Here we establish the corresponding invariance statement for the category of microsheaves on the open subset   
of stable Higgs bundles on nonstacky components where all semistables are stable, e.g. for coprime
rank and degree or for a punctured curve with generic parabolic weights.  
The proof uses the known global symplectic geometry 
of the Higgs moduli space to invoke recent results on the invariance of microlocal sheaves.  
 \end{abstract}


\maketitle




\section{Introduction}

Let $C$ be a complex algebraic curve and $G$ a reductive group.  We write
$Bun_{G}(C)$ for the moduli (stack) of $G$-bundles, possibly equipped with parabolic structures omitted from the notation, 
and $T^*Bun_{G}(C)$ for its cotangent space, often
identified with the moduli of Higgs bundles.  We recall that the cotangent to a smooth stack, while conic symplectic like
the cotangent bundle to a smooth scheme, also has somewhat surprising
features, being for instance only virtually smooth, and having scaling-fixed loci disjoint from the zero section \cite{Beilinson-Drinfeld}. 
Studying $G$-invariant functions yields a morphism $h: T^*Bun_{G}(C) \to A_{G, C}$ to a vector space \cite{Hitchin-system}; 
the central fiber $N = h^{-1}(0) \subset T^*Bun_{G}(C)$ is termed the global nilpotent cone.  

The  Betti variant of the geometric Langlands conjecture \cite{BenZvi-Nadler} has, 
on the automorphic or A-side, the category $Sh_N(Bun_{G}(C))$ of sheaves on $Bun_G(C)$ with microsupport in $N$.  
Because the conjecturally equivalent Galois, spectral or B-side concerns a category which depends 
on $C$ only through its topology, one  expects that $Sh_N(Bun_{G}(C))$ is locally constant
with respect to deformations of $C$.  This local constancy is not at all obvious, and is presently not known.
Indeed, a priori, the category $Sh_N(Bun_{G}(C))$ would be expected to depend sensitively 
on the singularities of the global nilpotent cone $N$, whose variation with $C$ is not easy to describe.

In general, given any smooth stack $M$ and any closed conic $V \subset T^*M$, one may consider $Sh_V(M)$, the category
of sheaves on $M$ which are microsupported in $V$.  There is a sheaf of categories $\mu sh_V$ 
on $T^*M$ (in fact, supported on $V$) such that $\mu sh_V(T^*M) = Sh_V(M)$ \cite{Kashiwara-Schapira}.   

There are so-called stable loci inside the moduli of bundles $Bun_{G}(C)$
and its cotangent $T^*Bun_{G}(C)$.  These are open sets, satisfying
$$T^*(Bun_{G}(C)^{st}) \subset  (T^*Bun_{G}(C))^{st} \subset T^*Bun_{G}(C)$$
where all inclusions are typically strict.  We may consider the restriction functor: 

\begin{equation} \label{eq: restriction to higgs}
\mu: Sh_N(Bun_{G}(C)) = \mu sh_N(T^*Bun_{G}(C)) \to  \mu sh_N((T^*Bun_{G}(C))^{st})
\end{equation}

We will write $Bun_G(C, \alpha)$ where $\alpha$ will mean either (a) discrete data determining
a choice of component of $Bun_G(C)$, or more generally (b) data of marked points and parabolic weights 
or even (c) marked points, irregular singularity types, and parabolic weights.  In particular, $\alpha$ contains
the data required to specify a stability condition.  

\begin{definition} \label{good}
We say $\alpha$ is good if the following hold: 

\begin{enumerate}
\item \label{not a stack} $(T^*Bun_G(C, \alpha))^{st}$ is a smooth manifold.
\item \label{hamiltonian} The scaling action of $S^1 \subset \CC^*$  is Hamiltonian for the natural K\"ahler form on $(T^*Bun_G(C, \alpha))^{st}$.
\item \label{proper bounded moment} The moment map for this Hamiltonian action is proper and bounded below.
\item \label{right skeleton} $N \cap (T^*Bun_G(C, \alpha))^{st} \subset (T^*Bun_G(C, \alpha))^{st}$ is the locus of points whose $\CC^*$-orbits have compact closure. 
\end{enumerate}
\end{definition}

Let us discuss when these conditions hold. Always the center $Z(G)$ sits inside the stabilizer of every point of $Bun_G(C)$ hence 
of $T^* Bun_G(C)$; to have a chance at getting a smooth manifold (rather than stack) we replace $Bun_G(C)$ by its rigidification 
with respect to $Z(G)$ as in \cite[Appendix A]{AOV}.
Now the stable locus will be generically non-stacky, and for some groups, e.g. $GL(n)$ or $SL(n)$, will have no 
orbifold points at all.\footnote{Our methods extend straightforwardly to the case when $(T^*Bun_G(C, \alpha))^{st}$ is globally a finite quotient of some 
space satisfying the hypotheses of Definition \ref{good}, compatibly with deformations of $(C,\alpha)$.} 
The remaining
conditions are known to hold in the unramified or tame cases 
whenever all semistables are stable, e.g. for $GL(n)$ bundles of degree $d$ coprime to $n$. 
The same is presumably true in the wild case, for which however the literature
 is presently underdeveloped. 

The purpose of this note is to observe: 

\begin{theorem} \label{thm: main}   
The category $\mu sh_N((T^*Bun_{G}(C, \alpha))^{st})$ is locally constant 
as we deform 
through good $(C, \alpha)$. 
\end{theorem}
\begin{proof} 
The hypotheses ensure that $(T^*Bun_G(C, \alpha))^{st}$ is
a Weinstein symplectic manifold with core $N$, as asserted in Corollary \ref{cor: is Weinstein} below. 
Microsheaf categories were shown to be invariant under Weinstein deformation in \cite{Nadler-Shende};
some additional technical work here is required to check that microlocalization in cotangent
bundles of stacks is compatible with the definition of microsheaves in \cite{Nadler-Shende}.  After carrying
this out, we arrive at an invariance result which directly applies in the present context in Proposition \ref{invariance}
below. 
This implies the stated result.
\end{proof}

Some properties of the map $\mu$ of \eqref{eq: restriction to higgs}, in particular that it is non-zero, follow from known results.
Indeed, recall that for any microsupport condition $\Lambda$ and any smooth Lagrangian point $p \in \Lambda$, there
is a `microstalk' functor from $(\mu sh_{\Lambda})_p$ to the coefficient category. 
In particular,  
on $Sh_N(Bun_{G}(C))$, the functor $f_\xi$ of microstalk at a smooth point $\xi \in N \cap (T^*Bun_G(C, \alpha))^{st}$ factors through $\mu$. 
For $\xi$ the intersection of $N$ with the Kostant section, the functor $f_\xi$ is known to be hom (up to a shift) from the Whittaker sheaf~\cite{NT}.
The pairing of the Whittaker sheaf and certain Eisenstein series is known to be non-zero, see~\cite{Taylor} and the citations therein. Thus $\mu$ is non-trivial on such Eisenstein series.
%
%


Given that $\mu sh_N((T^*Bun_{G}(C))^{st})$ is in fact locally constant under deformations of $C$, it is natural
to ask: 

\begin{question} \label{galois stable} 
Does the category $\mu sh_N((T^*Bun_{G}(C))^{st})$ and the map \eqref{eq: restriction to higgs} have a counterpart on the Galois side of Betti geometric Langlands?
\end{question} 

\begin{remark} 
In the present article, we use the methods of microsheaf theory, leaning on the invariance 
theorem of \cite{Nadler-Shende}.  However, the constancy of 
$\mu sh_N((T^*Bun_{G}(C, \alpha))^{st})$ can be seen from another point of view: 
for any (stably polarized) Weinstein manifold $W$, one has from \cite{GPS3} an isomorphism 
$\Gamma(W, \mu sh_{\mathfrak{c}_W}) \cong Fuk(W)$ between the category of microsheaves
supported on the core and the wrapped Fukaya category -- and the local constancy of $Fuk(W)$ 
under deformations of $W$ is essentially immediate from the local constancy of symplectic manifolds in families.\footnote{In particular, 
we could appeal to \cite{GPS3} instead of \cite{Nadler-Shende}. The former
depends on some results of the latter, but not the invariance theorem directly.}  

In terms of this interpretation, Question \ref{galois stable} might be restated as asking for a precise
relationship between the Betti geometric Langlands conjecture and homological mirror symmetry for moduli of stable Higgs bundles. 
(Many authors have considered related questions, e.g. \cite{Hausel-Thaddeus, Donagi-Pantev-duality, Kapustin-Witten}, but we are not aware
of any even conjectural answer to Question \ref{galois stable}.) 
One can also use the relationship with Fukaya categories to produce objects from 
smooth Hitchin fibers, which presently cannot be done with microsheaf methods alone, as these fibers are nonexact \cite{Shende-fibers}.  
\end{remark} 

\vspace{2mm}
\noindent {\bf Acknowledgements.}  
We thank Dima Arinkin,  David Ben-Zvi, Dennis Gaitsgory, Nick Rozenblyum and Zhiwei Yun for helpful discussions.  

DN was partially supported by NSF grant DMS-2101466.
VS is supported by the Villum Fonden (Villum Investigator 37814),  the Danish National Research foundation (DNRF157), 
the Novo Nordisk Foundation (NNF20OC0066298), and the USA NSF (CAREER DMS-1654545).

\section{Recollections from exact symplectic geometry} \label{sec: exact}
We review here some notions from exact symplectic geometry.  A textbook treatment is 
\cite{Cieliebak-Eliashberg}. 

A Liouville form on a manifold $M$ is a differential 1-form $\lambda$ such that $\omega = d\lambda$ 
is nondegenerate so a symplectic form.  Such a form determines a Liouville vector field $Z$ characterized by $\lambda = \omega(Z, \cdot)$
or equivalently $L_Z \omega = \omega$.  From the latter characterization, one sees that any $Z$-invariant subset must have 
zero or infinite volume, and in particular that $M$ cannot be a closed manifold (unless it is zero-dimensional).  

An 
exact symplectic (noncompact) manifold-without-boundary $M$ is said to be 
Liouville\footnote{In fact this is what is usually called a finite type Liouville manifold.  Here we will only be interested
in such finite type objects, so omit the adjective.}  if: 
\begin{enumerate}
\item $Z$ is complete 
\item \label{liouville lyapunov} there is a proper and bounded below smooth function $\phi: M \to \R$ which is Lyapunov for $Z$ and has no critical
points near infinity (= in the complement of a compact set). 
\end{enumerate}
If $M_\circ \subset M$ is a compact manifold with boundary and $Z$ is outwardly transverse to the boundary, then
we say $M_\circ$ is a Liouville domain.  In this case, $(\partial M_\circ, \lambda)$ is a contact manifold.  
In case $\phi$ has no critical points in $M \setminus M_\circ$, we say that $M$ is the completion of $M_\circ$; 
indeed one can reconstruct $M$ by extending $M_\circ$ in a canonical way.  The above hypotheses
guarantee that for $c \gg 0$, the space $M_{\le c} := \phi^{-1}(-\infty, c]$ is a Liouville domain completing to $M$. 
Finally, the {\em core} $\frc_M \subset M$ is the locus of all points which do not escape to infinity
under the Liouville flow.  The core is contained inside any Liouville domain completing to $M$. 

A family of exact symplectic manifolds $(M_t, \lambda_t)$ is said to be a family of Liouville manifolds if one can choose
$\phi$ defined (and smooth) on the total space of the family, for which each $\phi_t$ witness \eqref{liouville lyapunov} above.
In this case, by a Moser argument one can show the existence of a family of diffeomorphisms $A_t: M_0 \to M_t$ such that 
$A_t^{-1} \lambda_t =  \lambda_0 + dh_t$ for compactly supported $h_t$.  

A Liouville manifold was originally said to be Weinstein if $\phi$ can be chosen to be Morse.  This condition has been variously
weakened in subsequent literature; here we will allow $\phi$ to be Bott-Morse.  
A key feature is that the descending manifolds from the critical loci of $\phi$ are all isotropic, and $\frc_M$ is their union. 

A 1-parameter family of Liouville manifolds (as defined above) is typically said to be a family of Weinstein manifolds if the family $\phi_t$ 
is generalized Morse (i.e. behaves as a generic 1-parameter family of functions), or the analogous Bott-Morse notion.  For our applications
here, we will in fact only need to consider families which are genuinely Bott-Morse, i.e. the critical loci vary smoothly with $t$.  

We recall how Weinstein manifolds arise from conic holomorphic symplectic geometry: 

\begin{lemma} \cite{Frankel} \label{lem: Frankel} 
Let $X$ be a K\"ahler manifold with Hermitian metric $g + i \omega$.    
Assume given an $S^1$-action respecting the K\"ahler structure, which
is moreover Hamiltonian for $\omega$ with proper moment map $m: X \to \RR$.  Then: 
\begin{enumerate}
\item The function $m$ is a perfect Bott-Morse function, with critical locus the fixed-points $X^{S^1}$.
\item The $S^1$-action may be extended to 
a $\CC^*$-action where $\RR_{>0}$ acts by gradient flow for $m$.  
\end{enumerate} 
\end{lemma} 

\begin{proposition} \label{prop: weinstein from holomorphic symplectic}
In the situation of Lemma \ref{lem: Frankel}, assume given also a 
holomorphic symplectic form $\Omega$ on which $S^1$ acts by a nonzero character $d \in \ZZ$. 
Then $\frac{1}{d} \nabla m$ is a Liouville vector field for any of the real symplectic forms 
$\mathrm{Re}(e^{i \theta} \Omega)$.  

In particular, if in addition $m$ is bounded below, then 
$(X, \mathrm{Re}(e^{i \theta} \Omega))$ is a Weinstein manifold, with witness function $m$. 
The core $\frc_X$ is the locus of points with bounded $\CC^*$-orbits. 

Moreover, if given a smooth family $(X_t, g_t, \omega_t, \Omega_t)$ and smooth family of Hamiltonian $S^1$-actions, 
then  $(X_t, \mathrm{Re}(e^{i \theta} \Omega_t))$ is a family of Weinstein manifolds. 
\end{proposition}

\section{Recollections from geometry of moduli of Higgs bundles} \label{sec: geometry}

Some of the original references on the geometry of moduli of Higgs bundles include: \cite{Hitchin, Donaldson, Corlette, Hitchin-system, Simpson-higgs, Simpson-moduli-1, Simpson-moduli-2, Faltings, Laumon, Ginzburg}.  For curves with marked points and parabolic structures, 
valuable references include \cite{Simpson-noncompact, Biquard,  Konno-parabolic-higgs, Yokogawa-parabolic-higgs,
Yokogawa-deformation, Nakajima-parabolic-higgs, Thaddeus-parabolic-higgs}.

Recall we write $Bun_C(C, \alpha)$ for a component of the moduli stack of $G$-bundles on $C$, with various
discrete data collected into $\alpha$.  Then $T^*Bun_{G}(C, \alpha)$ is a (derived) symplectic stack; in particular, 
on its underived and non-stacky locus, it is smooth, and carries a complex structure $I$ and holomorphic symplectic form $\Omega_I$.  
The $\CC^*$ scaling action on the cotangent fibers preserves the stable locus, 
and acts with weight 1 on the line spanned by $\Omega_I$.

Here are some deeper properties, known to hold in at least the unramified and tamely ramified settings.  
The stable locus $(T^*Bun_{G}(C, \alpha))^{st}$ is a smooth orbifold.  In fact it carries a natural K\"ahler metric $g$; we write 
$\omega_I (v, w) := g(v, Iw)$ for the K\"ahler form, which is as always a real symplectic form.  We expand $\Omega_I := \omega_J + i \omega_K$; in fact
$\omega_J$ and $\omega_K$ are K\"ahler forms for other complex structures J and K, which together with $I$ provide a hyperk\"ahler structure.  
The $S^1 \subset \CC^*$ action is Hamiltonian for the symplectic form $\omega_I$.  Consequently its moment map
$m: (T^*Bun_{G}(C, \alpha))^{st} \to \RR$ is Bott-Morse, with critical locus equal to the $\CC^*$ fixed points.  
That is, $\alpha$ is good in the sense of Definition \ref{good}, save only for the possibility that $(T^*Bun_{G}(C, \alpha))^{st}$
may be an orbifold rather than a manifold.  

In any case, the conditions of Definition \ref{good} are precisely what is needed to apply 
Proposition \ref{prop: weinstein from holomorphic symplectic} to deduce: 

\begin{corollary} \label{cor: is Weinstein} 
For good $\alpha$, the space $T^*Bun_{G}(C, \alpha))^{st}$ is a Weinstein manifold with skeleton
$N \cap T^*Bun_{G}(C, \alpha)$. 
Deformations  through good $\alpha$ give Weinstein deformations of $T^*Bun_{G}(C, \alpha))^{st}$. 
\end{corollary}

\section{Microsheaves on stacks}

Let $M$ be a real manifold.  We write $Sh(M)$ for the category of sheaves on $M$ with coefficients in some 
appropriate symmetric monoidal stable presentable $\infty$-category $\mathcal{C}$, e.g. the dg derived category of  $\Z$-modules.  
In \cite{Kashiwara-Schapira} is introduced the notion of microsupport of sheaves: for each $F \in Sh(M)$,
a closed conic locus $ss(F) \subset T^*M$ of covectors along which sections of $F$ fail to propagate.  
For a conic subset $\Lambda \subset M$, one writes $Sh_\Lambda(M)$ for the full subcategory of sheaves
with microsupport contained in $\Lambda$.  

One obtains a sheaf of categories on  $T^*M$ by sheafifying the following
presheaf: 
\begin{equation} \label{microsheaves} 
\mu sh^{pre}_{T^*M}(U) = Sh(M) / \{F \,| \, ss(F) \cap U = \emptyset \}
\end{equation}
The objects of $\mu sh$ are termed microsheaves.  As $\mu sh$ is conic,
$\mu sh|_{T^*M \setminus M}$ descends to a sheaf of categories on the cosphere bundle $S^*M$,
which we denote also by $\mu sh$.  

For a closed subset $\Lambda$ of $T^*M$ or $S^*M$, we write
$\mu sh_\Lambda$ for the subsheaf of full subcategories of objects supported in $\Lambda$. 
When $\Lambda$ is smooth Lagrangian (or Legendrian), $\mu sh_\Lambda$ is locally isomorphic to the sheaf
of local systems on $\Lambda$.  The global twisting of $\mu sh_\Lambda$ is determined by the twisting of $\Lambda$ with
respect to the fiber polarization, and can be calculated by pullback from a universal map 
$\mu: U/O \to BPic(\cC)$.  

More generally, for any contact or conic symplectic manifold $V$ equipped with  a Maslov datum $\eta$, there is a sheaf of categories $\mu sh_{V, \eta}$, 
locally isomorphic to the corresponding sheaf of categories on a cotangent or cosphere bundle of the same dimension \cite{Shende-microlocal, Nadler-Shende}. 
About Maslov data: the symplectic tangent bundle or contact distribution is classified by a map $\tau: V \to BU$; a Maslov datum
is by definition a null-homotopy of the composition
$$V \xrightarrow{\tau} BU \to B(U/O) \xrightarrow{B\mu} B^2 Pic(\cC)$$

When $\cC$ is the dg derived category of  $\ZZ$-modules,  Maslov data
are canonically identified with pairs of a trivialization of the line bundle underlying $2 c_1(TV)$ (``grading") and a class in $H^2(V, {\pm 1})$ (``background class"). 
One can, and we will, always choose the trivial background class.  
(Precisely the same data is required to define various Floer theoretic invariants with $\Z$-grading and $\Z$-coefficients.) 
Moreover, {\em complex} symplectic or contact manifolds
carry canonical gradings (see e.g. \cite{perverse-microsheaves}), essentially because the complex symplectic group is simply-connected.  

A Lagrangian distribution (aka polarization) on $V$ determines a null-homotopy of $V \to B(U/O)$, hence in particular provides a Maslov
datum.  (More generally, we refer to a null-homotopy of $\gamma: V \to B(U/O)$ as a {\em stable polarization}.)  
Writing $\phi$ for the polarization by fibers of $T^*M$, one can show a canonical isomorphism 
\begin{equation} \label{only one mush} \mu sh_{T^*M} = \mu sh_{T^*M, \phi}
\end{equation} 
where the LHS is defined by \eqref{microsheaves} and the RHS is defined as in \cite{Nadler-Shende}.  

One of the main results of \cite{Nadler-Shende} is the following: 

\begin{theorem} \cite{Nadler-Shende} \label{mush invariance}
Let $W_t$ be a family of Weinstein manifolds, with cores $\Lambda_t$.  Let $\eta_t$ be a continuous family of Maslov data. 
Then the global sections category $\Gamma(W_t, \mu sh_{W_t, \eta_t})$ is locally constant in $t$.
\end{theorem}

Here we will require some generalizations to the case when $M$ is a smooth Artin stack (aka Lie groupoid in the differentiable category).  
All assertions regarding sheaves on such a stack can be interpreted in terms of sheaves on a simplicial space presenting the stack. 

First let us discuss the microsupport.  Consider a chart on $M$; i.e. some smooth manifold (or algebraic variety) $X$ and 
Lie (or affine algebraic) group $G$, such that $X / G \hookrightarrow M$.  Recall that there is a 
natural Hamiltonian $G$ action on $T^*X$; we denote the moment map $m: T^*X \to \mathfrak{g}^*$.  The quotient
$p: m^{-1}(0) \to m^{-1}(0) / G$ provides a chart $m^{-1}(0) / G \hookrightarrow T^*M$.  

 By definition, sheaves $F$ on $M$ restrict to $G$-equivariant sheaves $F_X$ on $X$.  Forgetting the equivariance, the 
microsupport of the underlying sheaf $F_X$ on $X$ is necessarily a $G$-invariant closed conic subset of $m^{-1}(0)$; in particular, it
descends to a closed conic subset in $T^*M$, which is easily checked to be independent of the chart.   This defines a notion of microsupport for sheaves $F$ on $M$;
having done so, the formula \eqref{microsheaves} defining microsheaves makes sense verbatim.  
We can compute in the local charts as follows: 

\begin{lemma} \label{cotangent equivariant mush} 
On $T^* (X/G) =   m^{-1}(0)/G$, there is a
 canonical equivalence of sheaves of categories: $\mu sh  \simeq (p_* \mu sh_{m^{-1}(0)})^{G}$,
 where $p:m^{-1}(0) \to m^{-1}(0)/G = T^*(X/G)$ is the $G$-quotient. 
\end{lemma}
\begin{proof}
 By construction, 
for open $U \subset T^*(X/G)$,  we have:
\begin{eqnarray*}
\mu sh^{pre} (U) & = & Sh(X/G)/ \{F \,| \, ss(F) \cap U = \emptyset \} \\
& = & Sh(X)^G/ \{F_X \,| \, ss(F_X) \cap p^{-1}(U) = \emptyset \} \\
& = &
Sh_{m^{-1}(0)}(X)^G / \{F_X  \,| \, ss(F_X) \cap p^{-1}(U) = \emptyset \}  \\
& = & \mu sh^{pre}(p^{-1}(U))^G  
\end{eqnarray*}
 After sheafification we have $\mu sh(U) = \mu sh^+(p^{-1}(U))^G$, where $\mu sh^+$ denotes the result of sheafifying $\mush^{pre}$ on $G$-invariant subsets of $T^*X$.   It remains to check on $G$-invariant subsets of $T^*X$, the $G$-invariants of the natural map of pre-sheaves $\mu sh^+ \to \mu sh$ is an equivalence. 
We give an inverse using the canonical resolution associated to the  comonad of $G$-equivariantization. Namely, given any section $F$ of $\mu sh$ over a $G$-invariant subset, its $G$-equivariantization $T(F)$ will be a section of $\mu sh^+$;  if   $F$ is $G$-equivariant, then repeating the procedure gives a resolution $[\cdots \to  T^2(F) \to T(F)] \to F$ by sections of $\mu sh^+$.    
\end{proof} 

Now consider some open $V \subset T^*(X/G)$ such that $V$ is non-stacky.  
(Such $V$ need not project to the non-stacky locus of $X/G$.)  We can define microsheaves on $V$ by 
$\mu sh_{T^*(X/G)}|_V$.  However, $V$ is itself an exact symplectic manifold, so given a Maslov datum
for $V$ we can also define microsheaves as in \cite{Nadler-Shende}.   We would like to explain now
(how to choose a Maslov datum so) that these notions agree.  The reason this is relevant to our aims is that 
on the one hand we have a restriction map $Sh(T^*(X/G)) \to \mu sh_{T^*(X/G)}(V)$, and on the other, 
we will want to appeal to Theorem \ref{mush invariance} which does not directly apply to $\mu sh_{T^*(X/G)}(V)$, 
but rather to $\mu sh_{V, \eta}(V)$ for Maslov datum $\eta$.

\vspace{2mm}
We first study descent of polarizations.  
Consider a symplectic manifold $W$ and a Lie group $G$  acting symplectically on $W$. 
Let $\pi:W \to W/G$ denote the projection to the stack quotient. Note that $TW$ is canonically $G$-equivariant, 
or equivalently, there is a vector bundle $\cV\to W/G$ together with an isomorphism $\pi^*\cV \simeq TW$. 
In particular, the classifying maps $\tau$ and $\gamma$ factor through $\pi$ composed with the respective maps 
$\ol \tau: W/G\to BU$ and $\ol \gamma: W/G \stackrel{\ol \tau}{\to} BU \to B(U/O)$ associated to $\cV$.

\begin{defn}
A $G$-equivariant stable polarization of $W$ is a $G$-equivariant lift of $\tau$ along $BO\to BU$ or equivalently a $G$-equivariant  trivialization of $\gamma$. (Here $G$ acts trivially on the classifying spaces $BU, BO, B(U/O)$.) 
\end{defn}

Note 
 a $G$-equivariant stable polarization of $W$ is equivalent to a lift of $\ol \tau$ along $BO\to BU$ or equivalently a  trivialization of $\ol \gamma$.

\begin{lemma} \label{stable polarization descent} 
Let $\wt W$ be a symplectic manifold. Let $G$ be a Lie group, and consider a Hamiltonian $G$ action on $\wt W$ with moment map $m:\wt W\to \frg^*$. 
Assume $0 \in \frg^*$ is a good value of $m$, in particular that $0\in \frg^*$ is a regular value, so that the reduction $W = m^{-1}(0)/G$ is also a smooth symplectic manifold. 

Then $G$-equivariant stable polarizations of $\wt W$ near $m^{-1}(0)$ are canonically identified with  stable polarizations of  $W$.
\end{lemma}

\begin{proof}
Note $m^{-1}(0)$ is a deformation retract of a neighborhood so $G$-equivariant stable polarizations of   $\wt W$ near $m^{-1}(0)$ are equivalent to 
$G$-equivariant stable polarizations of  $\wt W$ along $m^{-1}(0)$. In what follows, we  will work  with $G$-equivariant stable polarizations along $m^{-1}(0)$ 
in place of near $m^{-1}(0)$. Moreover,
we will then work along $W = m^{-1}(0)/G$ rather than $G$-equivariantly along $m^{-1}(0)$.

Along  $m^{-1}(0)$, we have a $G$-equivariant short exact sequence of vector bundles
\beq
\xymatrix{
\frg_{ m^{-1}(0)}  \oplus p^* TW \ar[r] & T\wt W|_{m^{-1}(0)} \ar[r] & \frg^*_{ m^{-1}(0)}
}
\eeq
where $\frg_{ m^{-1}(0)}$ denotes the trivial  bundle with fiber $\frg$ with its adjoint action,
and  $\frg^*_{ m^{-1}(0)}$ the trivial  bundle with fiber $\frg^*$ with its coadjoint action.
The inclusion of $\frg_{ m^{-1}(0)}$ 
 comes from the infinitesimal $G$-action and the projection to $\frg^*_{ m^{-1}(0)}$ is induced by the moment map $m$.

Thus along  $W = m^{-1}(0)/G$, we have a short exact sequence of vector bundles
\beq
\xymatrix{
TW \oplus \cV \ar[r] & (T \wt W|_{m^{-1}(0)})/G  \ar[r] &     \cV^*
}
\eeq
where $\cV$ and $\cV^*$ denotes the respective descents of $ \frg_{ m^{-1}(0)}$ and $\frg^*_{ m^{-1}(0)}  $. 
As with any such short exact sequence of topological bundles, we may (non-canonically) split it to obtain an isomorphism of vector bundles 
\beq
\xymatrix{
(T \wt W|_{m^{-1}(0)})/G \simeq TW \oplus \cV\oplus     \cV^*
}
\eeq

Consider the stable classifying maps $\wt \tau: W \to BU$, $\tau: W \to BU$, and $\nu: W \to BU$ 
for the respective bundles $(T \wt W|_{m^{-1}(0)})/G$,   $TW$, and $ \cV \oplus \cV^*$. 

We have $\wt \tau = \tau + \nu$. Since the polarization $\cV^* \subset \cV \oplus \cV^*$ gives a trivialization of $\nu$, 
we see trivializations of $\wt \tau$ are the same as trivializations of $ \tau$.
\end{proof}

\begin{corollary} \label{fiber polarization restriction}
Let $X$ be a manifold carrying the action of a Lie group $G$.  Let $V \subset T^*(X/G)$ be an open non-stacky
subset.  Then the fiber polarization on $T^*X$ descends canonically to a stable polarization on $V$. 

More generally, for any smooth stack $M$ and non-stacky $V \subset T^*M$, there is a canonical stable polarization
on $V$ given in charts as above. 
\end{corollary} 
\begin{proof}
We apply Lemma \ref{stable polarization descent} with $W = V$ and $\wt {W} = p^{-1}(V) \subset m^{-1}(0)$. 
\end{proof}
\begin{remark}
Note that by contrast it is not clear how to descend the fiber polarization on $T^*X$ to a polarization on $V$, since the fiber
polarization is not generally transverse to $\mu^{-1}(0)$. 
\end{remark}

We turn to microsheaves.  We begin by observing their functoriality under contactomorphisms respecting the Maslov datum.

\begin{lemma} \label{topological action} 
The group of Hamiltonian Maslov contactomorphisms or conic exact symplectomorphisms acts topologically on microsheaves.
\end{lemma}
\begin{proof}
We write here in the contact case, the conic exact symplectic setting follows by an analogous argument. 
For the notion of Hamiltonian contactomorphism, see e.g. \cite[2.3]{Geiges-contact};
note this requires $W$ to be co-oriented.  
So fix a co-oriented contact manifold $W$,
and let $H$ be the topological group of Hamiltonian contactomorphisms.  We regard it as a simplicial set
whose 0-cells are Hamiltonian contactomorphisms, 1-cells are 1-parameter families of Hamiltonian contactomorphisms, etcetera.

Fix Maslov data $\eta$ on $W$.  
We write $\widetilde{H}$ for the corresponding topological group whose 0-cells are elements $h \in H$ plus 
homotopies $h^*\eta \sim \eta$, etcetera. 

Let $A$ be the topological group whose 0-cells are a homeomorphism $a: W \to W$ and an 
isomorphism $a^* \mu sh_{W, \eta} \cong \mu sh_{W, \eta}$, 
whose 1-cells are homotopies of homeomorphisms $a: I \times W \to W$ 
along with isomorphisms $a^* \mu sh_{W, \eta} \cong \pi^* \mu sh_{W, \eta}$ 
where $\pi: I \times W \to W$ is the projection; etcetera. 

The lemma asserts that there is a morphism of topological groups $\widetilde{H} \to A$.  
On 0-cells this is given by the fact that $\mu sh_{W, \eta}$ is (up to contractible choices) canonically 
determined by $(W, \eta)$; so a contactomorphism $a$ and homotopy $\eta \to a^* \eta$ determines an isomorphism
$a^* \mu sh_{W, \eta} \cong \mu sh_{W, a^* \eta}$, functorial in composition of $a$.  

Now let us explain what to do on $1$-cells; morphisms on higher cells are constructed analogously.  
Let $H_t:I \times W \to \RR$ be the time-dependent Hamiltonian generating the  path of
 automorphisms  $g_t g_0^{-1}:I \to G$. Then the stabilization $W \times T^*I \simeq W \times I \times \RR$ admits
 the automorphism
$(w, t, \xi) \mapsto (g_t(w), t, \xi + H_t(w))$. It also admits the translation
automorphism
$(w, t, \xi) \mapsto (w, t, \xi - H_t(w))$. From the two, plus the data from the 1-cell of $\widetilde{H}$ of a homotopy
$g_t^* \eta \cong \pi^* \eta$, we deduce an isomorphism
$g_t^*\mu sh_{W, g_t^* \eta} \simeq \mu sh_{W, \pi^* \eta} \boxtimes Loc_{I}$ over $W\times I$.
\end{proof} 

\begin{remark}
In Lemma \ref{topological action}, we {\em do not} claim an action of the group of {\em eventually} conic Hamiltonian symplectomorphisms
in the exact symplectic (e.g. Weinstein) case.   The gapped specialization theorem of \cite{Nadler-Shende}, could be used to construct such an action on 
the global sections category of microsheaves supported in the core, in the Weinstein case, for sufficiently Weinstein Hamiltonian symplectomorphisms.  
\end{remark}

\begin{corollary}
If $W$ is contact co-oriented or exact symplectic, and $G$ is a Lie group of Hamiltonian contactomorphisms or exact conic symplectoorphisms,
and $\eta$ is $G$-equivariant Maslov data for $W$, then $\mu sh_{W, \eta}$ is $G$-equivariant.  In particular, 
there is a $G$ action on its global sections.  
\end{corollary} 

\begin{remark} 
It is clear from the construction that $W = T^*X$, the $G$ action is the canonical lift of a $G$ action on $X$, and the Maslov data induce from the 
(evidently $G$-equivariant) fiber polarization of $T^*X$, 
the action described
in Lemma \ref{topological action} is the microlocalization of the usual $G$-action on sheaves on $X$. 
\end{remark} 
 
Suppose now $W$ is contact or exact symplectic, and $p: P \to W$ is a principal $G$-bundle. 
As a special case of \cite[Sec. 10]{Nadler-Shende}, the relative cotangent bundle $\wt p:\wt W := T^*p \to W$ gives a
 $T^*G$ fibration which
  is correspondingly contact
or exact symplectic.   
Then a stable polarization or Maslov datum on $W$ pulls back to a 
$G$-equivariant corresponding such structure on $\wt W$.  
Then we have a canonical identification
$\mu sh_{P, \tilde{p}^* \eta} \simeq p^*\mu sh_{W, \eta}$, 
  the pushforward $p_* \mu sh_{P, \tilde{p}^* \eta}$ carries a $G$-action, and pullback induces a map 
of sheaves
$$
\mu sh_{W, \eta} \to (p_* \mu sh_{P, \tilde{p}^* \eta})^G
$$

\begin{lemma} 
For $p: P \to W$ as above and $\eta$ a Maslov datum for $W$, 
the natural map is an isomorphism
\begin{equation}
\label{equivariant mush} 
\mu sh_{W, \eta} \isom (p_* \mu sh_{P, \tilde{p}^* \eta})^G
\end{equation}  
\end{lemma} 
\begin{proof} 
 It suffices to check after pullback to $P$, i.e. that the induced map
 $$
\mu sh_{P, \tilde{p}^* \eta} \simeq p^*\mu sh_{W, \eta} \to p^*(p_* \mu sh_{P, \tilde{p}^* \eta})^G \simeq (p^*p_* \mu sh_{P, \tilde{p}^* \eta})^G
$$
is an isomorphism. By standard identities, we have 
  $$
(p^*p_* \mu sh_{P, \tilde{p}^* \eta})^G \simeq \mu sh_{P, \tilde{p}^* \eta}
$$
so that the map induced by \eqref{equivariant mush} is the identity.
\end{proof} 

\begin{corollary} \label{weinstein subset}
Let $M$ be a smooth stack and 
$U \subset T^*(X/G)$ be a nonstacky open subset.  Let $\eta$ be the Maslov datum on $U$ corresponding to the 
stable polarization from Corollary \ref{fiber polarization restriction}.  Then the natural map is an isomorphism 
$\mu sh_{T^*(X/G)}|_U \isom \mu sh_{U, \eta}$.   
\end{corollary} 
\begin{proof}
We may check locally after pullback to $T^*X$, applying   
Lemma \ref{cotangent equivariant mush} on one side and formula \eqref{equivariant mush} on the other 
in order to reduce to~\eqref{only one mush}. 
\end{proof} 

\begin{proposition} \label{invariance} 
Let $M$ be a smooth stack, and $\Lambda \subset T^*(M \times \R)$ a conic subset.  Suppose 
$U \subset T^*(M \times \R)$ be an open subset containing $\Lambda$.  Assume the symplectic reductions $U_t$ over $t \in \R$ 
determine a deformation of Weinstein manifolds with cores $\Lambda_t$.  Then $\mu sh_{T^*M}(\Lambda_t)$ is independent of $t$. 
\end{proposition}
\begin{proof}
Using Corollary \ref{weinstein subset}, the result follows from Theorem \ref{mush invariance}. 
\end{proof}

\bibliographystyle{plain}
\bibliography{refs}

\end{document}